\documentclass[a4paper, 12pt]{article}
\usepackage[utf8]{inputenc}
\usepackage{amssymb,amsmath,amsthm}
\usepackage{hyperref}
\usepackage[shortlabels]{enumitem}
\usepackage{authblk}

\usepackage{amsmath, amsthm, comment,xcolor}
\usepackage{amsfonts}
\usepackage{amssymb}
\usepackage{hyperref,cleveref}

\newtheorem{theorem}{Theorem}
\newtheorem{proposition}{Proposition}
\newtheorem{lemma}{Lemma}
\newtheorem{corollary}{Corollary}
\newtheorem{question}{Question}
\newtheorem{remark}{Remark}
\newtheorem{definition}{Definition}

\newcommand{\R}{\mathbb{R}}

\DeclareMathOperator{\dist}{dist}
\DeclareMathOperator{\Vol}{Vol}
\DeclareMathOperator{\diam}{diam}
\DeclareMathOperator{\length}{length}
\DeclareMathOperator{\Jac}{Jac}

\title{Volume growth of horospheres in diagonalizable Heintze groups}
\author{Gilles Courtois \thanks{Sorbonne Universit\'e and Universit\'e Paris Cit\'e, CNRS, 4 place Jussieu 75252 Paris Cedex 05; partially funded by ANR-23-CE40-0012-03
HilbertXField; gilles.courtois@imj-prg.fr} \quad Pablo Lessa \thanks{Centro de Matemática, Facultad de Ciencias, Universidad de la República, Igua 4225, 11400 Montevideo, Uruguay; lessa@cmat.edu.uy}  \quad Emiliano Sequeira \thanks{Centro de Matemática, Facultad de Ciencias, Universidad de la República, Igua 4225, 11400 Montevideo, Uruguay; partially founded by \emph{Fondo Vaz Ferreira-DICYT-MEC}, project FVF$\_2023\_522$; esequeira@cmat.edu.uy}}

\begin{document}
\maketitle

\begin{abstract}
We study the volume growth of horospheres in a Heintze group of the form $\R\ltimes_A\R^d$ with $A$ a diagonal derivation. We conclude that the isometry and quasi-isometry classes of horospheres (with their intrinsic geometry) coincide. 
Furthermore, if $A$ is not a scalar multiple of the identity, then there are exactly two such classes, characterized by their volume growth, which we calculate explicitly.    
\end{abstract}

\section{Introduction}

Simply connected manifolds of negative curvature have been thoroughly studied from the geometric and dynamical perspectives.   The horospheres of such a manifold $M$ comprise a family of hypersurfaces which may be viewed either as spheres centered at a point on the ideal boundary $\partial M$ of $M$ or as images by the projection $\pi : T^1M \to M$ of leaves of the strong stable foliation of the geodesic flow of $M$.

Very little is known about the \emph{intrinsic} geometry of the horospheres (see however \cite{1977heintze}), with the notable exception of symmetric spaces of negative curvature, where the horospheres, with their induced metric, are isometric to the Euclidean space or a $2$-step nilpotent Lie group with a left invariant metric. 

If \(M\) is the universal cover of a \emph{closed} negatively curved manifold of dimension at least $3$, the fact that {\it one} horosphere of $M$ is isometric to the Euclidean space implies that $M$ is isometric to the real hyperbolic space, up to an homothety,
\cite{BCH}. However, in the case $M$ does not cover a closed manifold, the existence of one horosphere isometric to the Euclidean space does not imply that $M$ has constant negative curvature even when \(M\) is homogeneous, as we will see.

\bigskip
The geometric property of the horospheres we are interested in is the volume growth of their induced metric. 
We say a complete \(C^1\) Riemannian manifold \((N,g_N)\) has volume growth of order \(r^k\) where \(k \ge 0\), if for each \(x \in N\) there exists \(C_x > 1\) and \(R_x > 0\) such that
\begin{equation}\label{defgrowth}
C_x^{-1}r^k \le \Vol\left(B(x,r)\right) \le C_xr^k,
\end{equation}
for all \(r > R_x\), where \(\Vol\) denotes the Riemannian volume and \(B(x,r)\) the ball of radius \(r\) centered at \(x\). Observe that to verify this property it is enough to show that \eqref{defgrowth} is true for some $x\in N$.

\smallskip
It is known that in a simply connected $(d+1)$-dimensional manifold $M$ with negative pinched sectional curvature 
$-a^2 \leq sec_M \leq -b^2 <0$, the horospheres have volume growth satisfying
$C^{-1} r ^{db/a} \leq {\rm Vol} B(x,r) \leq C r^{da/b}$. When there is a cocompact discrete group $\Gamma$ of isometries of $M$, the geodesic flow of $M$ is Anosov and the leaves of the strong stable foliation of the geodesic flow are dense in $T^1 (M \slash \Gamma)$, \cite[Theorem 15]{An}. Since the leaves are projected by $\pi : T^1 M \rightarrow M$ onto horospheres, 
one may expect that the recurrence properties of the leaves provide some kind of homogeneity of the horospheres and we formulate the following question.
\begin{question}\label{Question}
Let $M\slash \Gamma$ be a closed negatively curved manifold. Does there exists  \(k > 0\) such that every horosphere of $M$ has volume growth of order $k$?
\end{question}

There is a surjective mapping from the unit tangent bundle \(T^1M\) to the set of horospheres.  This mapping is defined by  $v\in T^1 M  \mapsto H(v)$, where $H(v)$ is the horosphere centered at $c_v (+\infty)$ and 
passing through $\pi (v)$, and $c_v$ is the geodesic ray such that $c_v (0) = \pi (v)$
and $\dot c _v (0) = v$.

Notice that it is not even known if, in the cocompact setting, the horospheres $H(v)$ have volume growth of order $k(v)$, a priori depending on $v$. However, assuming that such a function 
$v\in T^1 M  \to k(v)$ exists, then it has to be almost surely constant with respect to any geodesic flow invariant and ergodic measure on $T^1 M\slash \Gamma$. Indeed,
Let $\overline k : T^1 M \to \R$ be the function defined by 
$$\overline k (v) := \limsup _{r\to \infty} \frac{\log {\rm Vol} (B_{H(v)}(r))}{\log r}.$$
The function $\overline k$ is $\Gamma$-invariant and the induced function on the quotient,
$\overline{\mathcal K} (v) :T^1 M\slash \Gamma \to \R$, is almost surely constant with
respect to any ergodic measure on $T^1 M\slash \Gamma$. The same holds for 
$\underline{k} (v) := \liminf _{r\to \infty} \frac{\log {\rm Vol} (B_{H(v)}(r))}{\log r}.$
In particular, if for almost every $v\in T^1 M \slash \Gamma$, the horosphere $H(v)$ has polynomial 
growth of order $k(v)$, then $k(v)$ is almost surely constant.

\smallskip

The main goal of this paper is to study the growth of the horospheres of a family of homogeneous spaces of negative curvature. E. Heintze has studied the class of {\it homogeneous}
simply connected negatively curved manifolds. He showed that such manifolds are isometric to  solvable Lie groups endowed with a left invariant metric, \cite[Proposition 1]{1974heintze}.  He also gave a necessary and sufficient condition for a solvable Lie group to carry a left invariant negatively curved metric, \cite[Theorem 3]{1974heintze}. 

\smallskip

We will focus on the following particular case of these Heintze groups, defined as $G_A := \mathbb R \ltimes_A \mathbb R^d$, where $A$ is an $d\times d$ {\it diagonal} matrix with positive eigenvalues 
$\{\lambda _1 \leq \lambda _2\leq ....\leq \lambda _d\}$ and $\mathbb R$ acts on
$\mathbb R ^d$ by $(y,x) \in \R \times \R^d \rightarrow e^{yA} x$.  The metric on $G_A$ is defined as 
\begin{equation}\label{metricheintze}
g:= dy^2 + \sum _{i=1}^{d} e^{-2\lambda _i y} dx_i ^2.
\end{equation}
This defines a negatively curved Riemannian manifold (with sectional curvature bounded between $-\lambda_d^2$ and $-\lambda_1^2$, see Corollary \ref{CorCurvatura}), which we will call a \textit{real diagonal Heintze group}.

\medskip

Before stating our main theorem, let us recall a few important features of simply connected manifolds with negative curvature. Let $M$ denote a $(d+1)$-dimensional Riemannian manifold endowed with a metric of negative sectional curvature. It follows from the Cartan-Hadamard theorem that $M$ is diffeomorphic to $\mathbb{R}^{d+1}$. The geometric boundary $\partial M$ of $M$ is the set of equivalence classes of geodesic rays in $M$, where two geodesic rays are equivalent if they remain at a bounded Hausdorff distance. We recall that, in our context, it is homeomorphic to $\mathbb{S}^d$.

Fix a base point $x_0\in M$. For a unit tangent vector $v\in T_{x_0} M$ let $c_{v}$ be the unique geodesic ray determined by  $c_{v}(0)=x_0$ and $\dot c _{v}(0) =v$. It is well known that the map $v \in T_{x_0}  M \mapsto [c_{v}] \in \partial M$ defines a homeomorphism between the unit sphere in $T_{x_0} M$ and $\partial M$. Given a point $\xi = [c_{v}] \in \partial M$, the Busemann function 
(or horofunction) $b_\xi$ is then defined by 
$$b_{\xi}(x)= \lim_{t\to \infty} dist(x, c_{v} (t)) - dist(x_0, c_{ v} (t))\ \text{ for }x\in M.$$ 

When the sectional curvature of $M$ is negatively pinched, i.e., $-a^2 \leq sec \leq -b^2$, it is known that 
the Busemann function $b_\xi$ is $C^2$-smooth for every $\xi\in \partial M$,
\cite{1977heintze}.  Furthermore, 
for any $t\in\mathbb{R}$,
 the level set  $$H_{\xi}(t) =\left\{ x\in M : b_\xi(x)= t\right\}$$ is a $C^2$
submanifold of $M$ diffeomorphic to $\mathbb{R}^{d}$. We call it a {\it horosphere} centred at $\xi$. The sublevel set $$HB_{\xi}(t) =\left\{ x\in {M}: b_\xi(x)\leq  t\right\}$$ is called a {\it horoball} centered at $\xi$.
It follows that  horospheres inherit a complete $C^1$ Riemannian metric induced by the restriction of the metric of $M$. For instance, if 
$M$ is a real hyperbolic manifold,
every horosphere of $M$ is isometric to the Euclidean space $\R^{d}$.  If furthermore, the covariant derivatives of the curvature are uniformly bounded (which occurs for homogeneous Cartan-Hadamard manifolds, as well as for those covering a closed manifold), then horofunctions are of class \(C^\infty\) (see \cite{BCH2} and \cite{shcherbakov}).

Our main result is

\begin{theorem}\label{mainthm}
Let $G_A$ be a real diagonal Heintze group.
Assume \(\lambda_1 < \lambda_d\), then there exist exactly two isometry classes of horospheres in \(G_A\).
One class consists of horospheres isometric to the Euclidean space $\mathbb R^d$, and the other of horospheres which have volume growth of order \(r^{k}\) where
  \[k = \frac{\lambda_1 + \cdots + \lambda_d}{\lambda_1}.\]
\end{theorem}

We will show that if two horospheres are quasi-isometric, then they must have the same order of volume growth. Thus, Theorem 1 has the following consequence:

\begin{corollary}\label{quasiisom-cor}
Let $G_A$ be a diagonal Heintze group.
Assume \(\lambda_1 < \lambda_d\), then
there exist exactly two quasi-isometry classes of horospheres in $G_A$.    
\end{corollary}

\begin{remark}
If all eigenvalues of $A$ coincide, then $G_A$ has constant sectional curvature and all horospheres have volume growth of order $r^d$.  
\end{remark}

\begin{remark}\label{rmk2}
In the case $\lambda_1 < \lambda_d$, there is a point in $\partial G_A$ that is fixed by every quasi-isometry of $G_A$ (see \cite[Corollary 6.9]{Pan89}). This also follows from \cite[Corollary 1.5]{Xie1}. As a consequence of this property, the Heintze group $G_A$ has no compact quotient. In particular, Theorem \ref{mainthm} does not provide a negative answer to Question \ref{Question}. 
\end{remark}

\begin{remark}
It is shown in \cite[Example 3.3]{Pan89} that the conformal dimension of $\partial G_A\setminus \{\infty\}$ (where $\infty$ is a special point on the boundary at infinity) with the quasi-conformal structure defined by the parabolic visual metric with center $\infty$ is 
\[k = \frac{\lambda_1 + \cdots + \lambda_d}{\lambda_1}.\]
Since this structure is locally quasi-conformally equivalent to the structure given by visual metrics, the conformal dimension of $\partial G_A$ is also $k$.    
\end{remark}

\begin{remark}
The critical exponent for the $L^p$-cohomology of $G_A$ in degree $1$ is also equal to $k$ (see \cite[Theorem 4]{Pan88}) and is a quasi-isometry invariant, as the conformal dimension. 
\end{remark}


\bigskip
\noindent
{\bf Overview of the proof :}  
 The idea is to approximate a horosphere $H$ of $M$ by a simpler subset $\mathcal H$, defined as the union of \(d\) smooth submanifolds pasted along their boundary. Each piece of the approximate horosphere $\mathcal H$ is defined by an equation, and its volume growth can be estimated (see Proposition \ref{approxgrowthlemma}). 
 In order to compare the estimate of the volume growth of $H$ and the one of $\mathcal H$,  
 we rely on the work of Coulhon and Saloff-Coste, which states that two metric measured spaces have volume growth of the same order if they are locally doubling and \emph{isometric at infinity}. Being isometric at infinity means being quasi-isometric and also satisfying some local geometric homogeneity. The geometric homogeneity properties of the approximate horosphere $\mathcal H$ and the actual horosphere $H$ come down to the fact that the volume of balls of radius $r$
 are bounded from above and below independently of their centers (see Corollary \ref{corollarycontvol} and Lemma \ref{cont-vol-H}). The fact that the horosphere $H$ and the approximate horosphere $\mathcal H$ are quasi-isometric relies on their convexity and projections arguments (see Proposition \ref{quasi-isometry-prop}). 
 
\bigskip
\noindent
{\bf Acknowledgement:} We thank Gabriel Pallier for his interest and useful comments. 
 
 \section{Geometric preliminaries}
 
 In what follows we fix \(A\) as above and let \(M = G_A\). The points in $M$ will be written in coordinates $(y,x)=(y,x_1,\ldots,x_d)$. 
 
 The curves $t\mapsto\gamma_x(t)=(t,x)\in M$ are all geodesic that are asymptotic to the future and defines different points to the past (we call them \textit{vertical geodesics}). 
We then get a complete description of $\partial M \approx \mathbb{S}^d = \mathbb R ^d \cup \{\infty \}$: the point $\infty$ is $\gamma_x(+\infty)$ for any $x\in \R^d$, and each point $x\in \R^d$ represents $\gamma_x(-\infty) \subset \partial M$.

\smallskip
In the whole paper, we will denote by $\xi _{+}$ the point $\infty$ and by $\xi _{-}$ the point $\gamma_0(-\infty)$.  We will write \(\dist(p,q)\) for the Riemmanian distance between $p$ and $q$ in $M$ and use \(\dist_X\) for any induced distance on a Lipschitz submanifold \(X\). 

\smallskip

We fix the base point \(o = (0,0)\) and recall that the Busemann function associated to a point \(\xi \in \partial M\) is defined by
\[b_{\xi}(p) = \lim\limits_{s \to +\infty}\dist(\alpha(s),p) - s,\]
where \(\alpha\) is the geodesic ray from \(o\) to \(\xi\).
Noting that 
\[\dist\bigl((s,x),(s,0)\bigr) \le e^{-\lambda_1 s} \left(\sum\limits_{i = 1}^d x_i^2\right)^{\frac{1}{2}},\]
we obtain $b_{\xi_+}(y,x) = -y.$ Hence the horospheres associated to the point \(\xi_+\) are the sets 
$$
H_{\xi_{+}} (t)=\lbrace t \rbrace \times \R^d,
$$
for $t\in \mathbb R$. It is immediate that they are isometric to \(\R^d\) with the standard Euclidean metric via the isometry \((t,x) \mapsto e^{-tA}x\), i.e., 
the induced metric on $H_{\xi_{+}} (t)$ writes
$$h_t := \sum _{i =1} ^{d} e^{-2\lambda _i t} dx_i ^2.$$ 
We call them \textit{horizontal horospheres} or \textit{Euclidean horospheres}.

The set $\{ \tau_\lambda : \lambda \in \R\}$ of transformations of $M$ defined by 
$$\tau_\lambda (t,x) = \bigl(t+ \lambda, e^{2\lambda A}x\bigr)$$ 
is a one parameter group of isometries of $M$, which acts transitively on the family of horizontal horospheres. On the other hand, the set of translations 
$\{T_z : z\in \R^d\}$ defined by 
$$T_z (t, x) := (t, x + z)$$ 
forms a group of isometries of $M$ which fix $\xi_{+}$ and globally preserves $H_{\xi_{+}} (t)$ for every $t\in \mathbb R$. Moreover, it acts transitively on $\partial M \setminus \{\xi_{+}\}$. Using successively elements $\tau_\lambda$ and $T_z$, one then deduce that all horospheres centered at points $x\in \R^d \subset \partial M$ belong to the same class of isometry. Since the group $\mathbb G_A$ generated by $\{ g_\lambda  : \lambda \in \R\}$ and $\{T_z : z\in \R^d\}$
acts isometrically and simply transitively on $M$, so that $M$ is isometric to $\mathbb G_A$ endowed with a left invariant metric. We add that when $\lambda _1 < \lambda _d$, every isometry of $M$ fixes the point $\xi _+$, cf. Remark \ref{rmk2}.  We sum up the above discussion in the following proposition.

\begin{proposition}\label{two-isom-class}
There are two orbits of the action of $\mathbb G _A$ on the set of horospheres of \(M\), i.e., the set of Euclidean horospheres centered at $\xi _+$, and the set of horospheres centered at $x\in \mathbb R^d$. 
\end{proposition}

We end these preliminaries with a description of the Levi-Civita connection of $M$ which will be useful later. We set $Y:= \partial _y$ and $X_i := \partial _{x_i}$ the vector fields associated to coordinates $(y,x_1,\ldots,x_d)\in M=\mathbb R \ltimes_A \mathbb R^d$.

\begin{lemma}\label{levi-civita}
Let $\nabla$ be the Levi-Civita connection of the metric 
$$g= dy^2 + \sum _{i=1}^{d} e^{-2\lambda _i y} dx_i ^2$$ on $M$. Then, at every point $(y,x) \in M$
\begin{enumerate}
\item[(i)]
$\nabla _Y Y =0$
    \item[(ii)] 
    $\nabla _Y X_i = 
    \nabla _{X_i} Y= -\lambda _i X_i$

    \item[(iii)] 
    $\nabla _{X_i}X_j =0$ for $i\neq j$

    \item[(iv)] 
    $\nabla _{X_i} X_i = \lambda _i e^{-2\lambda _i y} Y.$
\end{enumerate} 
\end{lemma}
\begin{proof}
Since $y \mapsto (y,x)$ is a geodesic, we have $\nabla _Y Y =0$, which prove $(i)$.

Notice that, for every $(u_j)_{j\neq i}$, the set 
$$\bigl\{(y,x) \in M : x_j = u_j\in \mathbb R, \, \forall j\neq i \bigr\}$$
is totally geodesic thus, $\nabla _Y X_i$ and $\nabla _{X_i}X_i$ are perpendicular to $X_j$ for every $j\neq i$, i.e., decomposes along $Y$ and $X_i$. We first observe that $$g\bigl(\nabla _Y X_i,Y\bigr) =- g\bigl(X_i, \nabla _Y Y\bigr) =0.$$
Since $g\bigl(X_i, X_i\bigr) = e^{-2\lambda _i y}$, we have 
\begin{equation}\label{Y-norm-X}
2 g\bigl(\nabla _{Y} X_i, X_i\bigr)=Y.g\bigl(X_i,X_i\bigr) = -2\lambda _i e^{-2\lambda _i y}
\end{equation}
which implies with the previous equality that $\nabla _Y X_i = -\lambda _i X_i$, which proves $(ii)$. 

Since $g\bigl(X_i,X_i\bigr) = e^{-2\lambda _i y}$ does not depend on $x$, we have 
$$2g\bigl(\nabla _{X_i} X_i, X_i\bigr)=X_i \,g\bigl(X_i, X_i\bigr) =0$$
hence, $\nabla _{X_i} X_i = a Y$ for some $a\in \mathbb R$. We have 
$$
g\bigl(\nabla _{X_i} X_i, Y\bigr)= 
-g\bigl(X_i,\nabla _{X_i} Y\bigr)=
-g\bigl(X_i,\nabla _{Y} X_i\bigr)=
\lambda _i e^{-2\lambda _i y},
$$
the last equality coming from (\ref{Y-norm-X}) and therefore, 
$\nabla _{X_i} X_i = \lambda _i e^{-2\lambda _i y} Y$ which proves $(iv)$.

It remains to prove $(iii)$. By $(ii)$ and $(iv)$, we have 
$$g\bigl(\nabla _{X_i}X_j ,Y\bigr) = \lambda _i g\bigl(X_j,X_i\bigr) =0 $$
and 
$$
g\bigl(\nabla _{X_i}X_j, X_i\bigr) = 
-\lambda _i e^{-2\lambda _i y}g\bigl(X_j, Y\bigr) =0.
$$
We also have $g\bigl(\nabla _{X_i}X_j, X_j\bigr)=0$ as $g\bigl(X_j, X_j\bigr) = e^{-2\lambda _i y}$ does not depend on $x$.  This ends the proof of (iii). 
\end{proof}

A direct calculation using the above lemma gives bounds for the curvature of $M$ (see also \cite[Proposition 6 (1)]{Garcia-Lara}). 

\begin{corollary}\label{CorCurvatura}
The sectional curvature of $M$ is bounded between $-\lambda_d^2$ and $-\lambda_1^2$.    
\end{corollary}
 
\section{Approximating horosphere}


In this section we construct a set $\mathcal H$ approximating the horosphere $H_{\xi_-}(0)$ and compute the order of its volume growth. This set will be defined as a kind of horosphere for a function $\rho:M\times M\to [1,+\infty)$ that approximates  the Riemannian distance of $M$ at large scale.

\subsection{Approximate distance}

Let \(\rho:M \times M \to [1,+\infty)\) be defined by
\[\rho\left((y,x),(y',x')\right) = 1 + |y-\hat{y}| + 2\inf\left\lbrace t \ge 0: \sum\limits_{i = 1}^d e^{-2\lambda_i (t+\max(y,y'))}(x_i-x'_i)^2 \le 1\right\rbrace.\]

Observe that, if the points $p$ and $q$ are sufficiently far apart, then $\rho(p,q)$ represents the length of the curve from $p$ to $q$. This curve goes through the vertical line to a point $\tilde{p} = (\tilde{y}, x)$, then follows the geodesic in $H_{\xi_+}(\tilde{y})$ from $\tilde{p}$ to a point $\tilde{q} = (\tilde{y}, \hat{y})$, and then continues along the vertical line from $\tilde{q}$ to $q$, where the distance in $H_{\xi_+}(\tilde{y})$ between $\tilde{p}$ and $\tilde{q}$ is $1$. This curve is intended to approximate the geodesic segment between $p$ and $q$.

\begin{proposition}\label{approximatedistance}
There exists \(C_\rho > 0\) such that
\[\dist(p,q) \le \rho(p,q) \le \dist(p,q) + C_\rho,\]
for all \(p,q \in M\).
\end{proposition}

Notice that, since $(M,g)$ is Gromov-hyperbolic the following property holds: there exists $\delta>0$ such that for every geodesic triangle $\Delta$ and the map $f_\Delta:\Delta\to T$ to a tripod $T$ that is an isometry on the edges, one has $\diam\bigl(f_\Delta^{-1}(w)\bigr)\leq \delta$ for every $w\in T$.

In order to prove Proposition \ref{approximatedistance} we will prove first the following lemmas, in which we write for convenience \(h(y,x) = y\) for  \((y,x) \in M\).

\begin{lemma}\label{emilemma1}
Let $\delta$ be as above, then for all $p,q\in M$ there exist $\tilde{p},\tilde{q}\in M$ such that
\begin{enumerate}
    \item[i)] $h(\tilde{p})=h(\tilde{q})$.

    \item[ii)] $p$ and $\tilde{p}$ are in the same vertical geodesic. The same for $q$ and $\tilde{q}$.

    \item[iii)] $\dist(p,q)=\dist(p,\tilde{p})+\dist(q,\tilde{q})$.

    \item[iv)] $\dist(\tilde{p},\tilde{q})\leq \delta$.  
\end{enumerate}
\end{lemma}

\begin{proof}Parameterize the geodesic ray $[q,\xi_+)$ by $\beta:[0,\xi_+)\to M$
and consider the geodesic triangle
$$\Delta(t)=[p,q]\cup [p,\beta(t)]\cup [q,\beta(t)].$$
Moreover, for each $t\geq 0$ consider $f_\Delta$ the corresponding map to a tripod $T(t)$ and the points $a(t)\in [p,\beta(t)]$, $b(t)\in [q,\beta(t)]$ and $c(t)\in [p,q]$ the preimages of the center of $T(t)$. Recall that $$\diam \{a(t),b(t),c(t)\}\leq \delta.$$

Observe that $a(t)$ and $b(t)$ belong to the ball of center $q$ and radius $2\dist(p,q)$, therefore there exists a sequence $t_n\to+\infty$ such that $a_n=a(t_n)\to \tilde{p}$ and $b_n=b(t_n)\to \tilde{q}$. Let us prove that these points satisfy conditions $(i)-(iv)$:
\begin{enumerate}
    \item[\textit{i)}] Since $a_n$ and $b_n$ are preimages of the center of $T(t_n)$, the sequence $$\dist(p,\beta(t_n))-\dist(q,\beta(t_n))=\dist(p,a_n)-\dist(q,b_n)$$ converges to $$\dist(p,\tilde{p})-\dist(q,\tilde{q}).$$
We therefore have $\dist(p,\tilde{p})-\dist(q,\tilde{q})= b_{\xi_+}(p) - b_{\xi_+}(q)$, and hence $\tilde{p}$ and $\tilde{q}$ are on the same horosphere centred at $\xi_+$, that is, $h(\tilde{p})=h(\tilde{q})$.

    \item[\textit{ii)}] It is clear that $\tilde{q}$ are in $[q,\xi_+)$. Moreover, since the edge $[p,\beta(t_n)]$ tends to $[p,\xi_+)$, then $\tilde{p}\in [p,\xi_+)$.

    \item[\textit{iii)}] Since $\dist(p,q)=\dist(p,c_n)+\dist(q,c_n)=\dist(p,a_n)+\dist(q,b_n),$ 
    we get, taking limit,
    $$\dist (p,q)= \dist(p,\tilde{p})+\dist(q,\tilde{q}).$$

    \item[\textit{iv)}] It is clear that $\dist(a_n,b_n)\leq \delta$ and thus $\dist(\tilde{p},\tilde{q})\leq \delta$.
\end{enumerate}
\end{proof}

Given $p=(y,x)\in M$, the Riemannian norm of $(v,u) =(v,u_1,\ldots,u_n) \in T_p M \simeq \mathbb R \times \mathbb R^d$ is
$$\|(v,u)\|=\|(v,u)\|_p = \bigl(v^2 +|u|^2_y \bigr)^{\frac{1}{2}},$$
where
$$|u|_y:= \left(\sum\limits_{i = 1}^d e^{-2 \lambda_i y}u_i ^2\right)^{\frac{1}{2}}.$$
For points $p = (y, x)$ and $\hat p = (y, \hat x)$ we write $|p-\hat{p}|_y=|x-\hat{x}|_y$. Observe that $|p-\hat{p}|_y=\dist_{H_{\xi_+(y)}}(p,\hat{p})$.

\begin{lemma}\label{emilemma2}
Let $p,p',q,q'\in M$ such that:
\begin{itemize}
    \item $p'\in [p,\xi_+)$ and $q'\in [q,\xi_+)$.

    \item $h(p')=h(q')=y$.

    \item $|p'-q'|_{y}=1$.
\end{itemize}
Then $\dist(p,p')+\dist(q,q')\leq \dist(p,q)+C$, where $C$ is a constant that does not depend on the points.
\end{lemma}

\begin{proof}
Take $\tilde{p}$ and $\tilde{q}$ as in Lemma \ref{emilemma1}. If $h(\tilde{p})\geq h(p')$, then
$$\dist(p,q)=\dist(p,\tilde{p})+\dist(q,\tilde{q})\geq \dist(p,p')+\dist(q,q').$$
To solve the other case we look for a bound for $s=\dist(\tilde{p},p')=\dist(\tilde{q},q')>0$, then we will have
$$\dist(p,p')+\dist(q,q')\leq \dist(p,\tilde{p})+\dist(q,\tilde{q})+2s=\dist(p,q)+2s.$$

Theorem 4.6 in \cite{1977heintze} says that, if $\bar{p}$ and $\bar{q}$ belong to the same horosphere $H$ of $M$, then
\begin{equation}\label{HIF}
\frac{2}{\lambda_1}\sinh\left(\frac{\lambda_1}{2}\dist(\bar{p},\bar{q})\right)\leq \dist_H(\bar{p},\bar{q})\leq \frac{2}{\lambda_d}\sinh\left(\frac{\lambda_d}{2}\dist(\bar{p},\bar{q})\right).
\end{equation}
Remember that the sectional curvature of $M$ is bounded between $-\lambda_d$ and $-\lambda_1$.

Using \eqref{HIF} and that the traslation $(y,x)\mapsto(y+s,x)$ has operator norm $e^{-\lambda_1s}$ we obtain 
$$1\leq e^{-\lambda_1s}|\tilde{p}-\tilde{q}|_{\tilde{y}}\leq e^{-\lambda_1s}\frac{2}{\lambda_d}\sinh\left(\frac{\lambda_d}{2}\delta\right),$$
and hence,
$$s\leq \frac{1}{\lambda_1}\log\left(\frac{2}{\lambda_d}\sinh\left(\frac{\lambda_d}{2}\delta\right)\right).$$

\end{proof}

\begin{proof}[Proof of Proposition \ref{approximatedistance}]
The first inequality direct because $\rho(p,q)$ is, at least, the length of a curve joining $p$ and $q$. We need to prove the other one.

Suppose (without loss of generality) that
$h(p)\leq h(q)$. We set $y:= h(q)$ and $\bar p$ the intersection of $H_{\xi_+}(y)$ and the geodesic ray $[p,\infty)$.

We distinguish two cases: 
\begin{enumerate}
    \item $|q-\bar p|_y\geq 1$: since the vertical translation decreases the distance on the horizontal horospheres, there are points $p'\in [p,\xi_+)$ and $q'\in [q,\xi_+)$ satisfying the assumptions of Lemma \ref{emilemma2} and therefore 
    $$\rho(p,q) = d(p, p') + d(q,q')) +1 \leq d(p, q) +C +1.$$
    \item $|q-\bar p|_y\leq 1$: in this case we have 
    $\rho(p,q)=\dist(p,\bar p)+1$; hence, by the triangle inequality, we get  
    $\rho(p,q)=\dist(p,\bar p)+1\leq \dist(p,q)+2$.
\end{enumerate}
\end{proof}

\subsection{Approximate horosphere}
The approximate horosphere we are looking for will be obtained as the boundary of an approximate horoball defined by 
\begin{equation*}\label{setV}
V = \left\lbrace (y,x) \in M:  \max\limits_{1 \le i \le d} e^{-\frac{\lambda_i y}{2}}|x_i| \le 1\right\rbrace.
\end{equation*}
Then we set $\mathcal{H} := \partial \bigl(V \cap \{y\leq 0\}\bigr).$

Observe that \(V \cap \lbrace y \le 0\rbrace\) is the union over  \(t > 0\) of the balls centered at \((-t,0)\) of radius \(t\) for the approximate distance \(\rho\). 

In the following Lemma, we show that the set $V\cap \{y\leq 0\}$ lies in between two nearby horoballs of $M$ and consequently, the approximate horosphere $\mathcal H$ lies between two nearby horospheres. 

\begin{lemma}\label{approxhorolemma}
 There exists a constant \(C_V > 0\) such that
   \[HB_{\xi_-}(-C_V) \subset V \cap \lbrace y \le 0\rbrace \subset HB_{\xi_-}(C_V).\]  
 In particular, the approximate horosphere $\mathcal H$ lies between the two horospheres $H_{\xi_-}(-C_V)$ and 
$H_{\xi_-}(C_V)$.
\end{lemma}

\begin{proof}
  Let \(\rho\) and the constant \(C_\rho\) be given by Proposition \ref{approximatedistance}.

  We first observe that for every $(y,x) \in M$,
  \[\lim\limits_{s \to +\infty}\rho\bigl((y,x),(-s,0)\bigr) - s = 1+ y + 2 R(y,x),\]
  where 
  \[R(y,x) = \inf\left\lbrace t \ge 0: \sum\limits_{i = 1}^d e^{-2\lambda_i (t+y)}x_i^2 \le 1\right\rbrace.\]
 
  This implies, by Proposition \ref{approximatedistance}, that
  \begin{equation}\label{approxbusemann}
  -C_\rho + 1 + y + 2R(y,x) \le b_{\xi_-}(y,x) \le C_\rho + 1+ y + 2R(y,x).
  \end{equation}
Assume first that \(b_{\xi_-}(y,x) \le 1 - C_\rho\).  Inequality \eqref{approxbusemann} then implies
  \[y + 2R(y,x) \le 0,\]
and, since \(R(y,x) \ge 0\), we get \(y \le 0\) and
  \[R(y,x) \le -\frac{1}{2}y.\]
By definition of \(R(y,x)\), setting \(t = -\frac{1}{2}y\)  we obtain
\[\sum\limits_{i = 1}^d e^{-2\lambda_i(t + y)}x_i^2 = \sum\limits_{i = 1}^d e^{-\lambda_iy}x_i^2 \le 1,\]
hence
\[\max\limits_{1 \le i \le d}e^{-\lambda_i y/2}|x_i| \le \left(\sum\limits_{i = 1}^d e^{-\lambda_iy}x_i^2\right)^{\frac{1}{2}} \le 1,\]
and therefore
\[HB_{\xi_-}(1-C_\rho) \subset \lbrace y \le 0\rbrace \cap V\]

Assume now that \((y,x) \in \lbrace y \le 0\rbrace \cap V\) and notice that this implies
\[\max\limits_{1 \le i \le d}e^{-\lambda_i y /2}|x_i| \le 1,\]
choosing \(t > 0\) so that \(e^{-\lambda_1(2t + y)} = \frac{1}{d}\) we have
\begin{align*}\sum\limits_{i = 1}^d e^{-2\lambda_i (t+y)}x_i^2  & =  \sum\limits_{i = 1}^d e^{-\lambda_i (2t+y)}\left(e^{-\lambda_i y/2}|x_i|\right)^2  
\\ &= \sum\limits_{i = 1}^d e^{-\lambda_i (2t+y)} \le 1.
\end{align*}

This shows that \(y + 2R(y,x) \le \lambda_1^{-1}\log(d)\), so by \eqref{approxbusemann} we obtain
\[\lbrace y \le 0\rbrace \cap V \subset HB_{\xi_-}\bigl(1+ C_\rho + \lambda_1^{-1}\log(d)\bigr).\]
This concludes the proof choosing 
$$C_V := \max \left\{ C_\rho -1, C_\rho +1 + \frac{\log d}{\lambda _1}\right\}.$$
\end{proof}

The following lemma will be important to prove that the approximate horosphere is quasi-isometric to the actual horosphere of $M$.
\begin{lemma}\label{vconvexitylemma}
The set \(V\) is convex.
\end{lemma}

\begin{proof}
Set $V_i^\pm := 
\left\lbrace (y,x) \in M:  
\exp(\lambda_i y/2) \pm x_i \ge 0\right\rbrace.$ Since 
$$V = \bigcap _i V_i^\pm,$$
the
convexity of \(V\) then follows immediately from the convexity of the $V_i^\pm$'s.

 Let us show that $V_i^\pm$ is convex for every $1\leq i \leq d$.
 By \cite[Corollary 1.2]{BCGS}, the convexity of $V_i^\pm$ is equivalent to the fact that the boundary $\partial V_i^\pm$ has non negative second fundamental form with respect to the interior normal.
Notice that, changing the sign of \(x_i\) is an isometry of \((M,g)\), hence it is sufficient to prove that $V_i^-$ is convex for every $1\leq i \leq d$. 
Fix \(i\) and let \(f(y,x) = \exp(ay) -x_i\).  We claim that the hypersurface  \(\lbrace f = 0\rbrace\) has non negative second fundamental form with respct to the interior normal of the domain $\{f \geq 0\}$ if \(0 \le a \le 2\lambda_i\). Taking $a= \lambda _i /2$ will then implies that $V_i^-$ is   convex. 

In order to prove the claim, we consider a curve \(\alpha(t) = (y(t),x(t))\) such that \(x_i(t) = \exp(ay(t))\) for all $t$ and we check that for any such curve, one has
$$df\left(\nabla_{\alpha'}\alpha'\right) \ge 0.$$

Let us compute $\nabla_{\alpha'}\alpha'$. We write
$$
\alpha '(t) = y'(t) \partial _y +
\Sigma _{j=1}^d x'_j (t) \partial _{x_j}.
$$
Setting $Y:= \partial _y$, $X_j := \partial _{x_j}$, we have 
\begin{eqnarray*}
\nabla_{\alpha'}\alpha' &=&
y'' Y+ y' \nabla _Y Y + \sum_{j=1}^d x''_j X_j +
\sum_{j=1}^d y' x'_j  \nabla _Y X_j
\\
&+&\sum_{i=1}^d y'x'_j \nabla _{X_j} Y+
\sum_{j=1}^d x'_j \nabla _{X_j} X_j\\
&=& y'' Y+ y' \nabla _Y Y+  \sum_{j=1}^d x''_j X_j +2
\sum_{j=1}^d y' x'_j  \nabla _Y X_j\\
&+&\sum_{i,j=1}^d x'_i x'_j \nabla _{X_i} X_j.
\end{eqnarray*}
By Lemma \ref{levi-civita}, we deduce


\[\nabla_{\alpha'}\alpha' = \left(y'' + \sum\limits_{j = 1}^{d}\lambda_jx_j'^2 e^{-2\lambda _jy}\right)\partial_y +  \sum\limits_{j = 1}^d (x_j'' - 2\lambda_j x_j'y')\partial_{x_j}.\]

So that, since $x'_i = a  y' \,e^{at}$, $x''_i = (
a^2 y'^2 + ay'')e^{ay}$ and $df = ae^{ay} -dx_i$

\begin{align*}
  df\left(\nabla_{\alpha'}\alpha'\right) 
  &= a e^{ay}\left( y'' + \sum\limits_{j = 1}^d \lambda_j x_j'^2 e^{-2\lambda_j y}\right) - x_i'' + 2\lambda_i x_i'y'  \\
&=  ae^{ay}y'' + a e^{ay}\sum\limits_{j = 1}^d \lambda_j x_j'^2 e^{-2\lambda _j y}- (a^2e^{ay}y'^2 + ae^{ay}y'') + 2 a \lambda_i e^{ay}y'^2  \\
&\ge (2a\lambda_i - a^2)e^{ay}y'^2 \ge 0
\end{align*}
if \(0 \le a \le 2\lambda_i\). This proves that the second fundamental form of $\{f =0\}$ is positive and conclude the proof of 
\end{proof}

\subsection{Volume growth of the approximate horosphere}
In this section, we study the volume growth of the approximate horosphere $\mathcal H$. The idea is to compare the balls of $\mathcal H$ with domains whose volume is easy to compute. 

It follows from the of the approximate horosphere that it can be descomposed into $2d+1$ faces:
\begin{equation}\label{DescH}
\mathcal H = \mathcal H_0 
\cup\left(\bigcup_{1\leq i \leq d} \mathcal H_i^{\pm}\right),    
\end{equation}
where 
$$
\mathcal H_0 = \left\lbrace (y,x)\in M :  y = 0\text{ and }\max\limits_{1 \le i \le d}e^{-\lambda_i y/2}|x_i| \leq 1\right\rbrace
$$
and
$$
\mathcal H_i^{\pm} =\left\lbrace (y,x)\in M : y \le 0\textcolor{red}{, }\max\limits_{1 \le j \le d, j\neq i}e^{-\lambda_j y/2}|x_j| \leq 1 \text{ and} \,\,e^{-\lambda_i y/2}x_i =\pm 1\right\rbrace.
$$
We notice that, since \(\mathcal H\) is the union of smooth submanifolds pasted along their boundary, it inherits the structure of a length space with a well defined volume.  Therefore the Definition \ref{defgrowth} of volume growth above extends to \(\mathcal H\).

\begin{proposition}
\label{approxgrowthlemma}
With respect to the inherited length structure and volume, the approximate horosphere \(\mathcal H\) has volume growth of order \(k\) with
\[k = \frac{\lambda_1 + \cdots + \lambda_d}{\lambda_1}.\] 
\end{proposition}

Before proceeding to the proof of Proposition \ref{approxgrowthlemma}, we introduce some notations and state some lemmas.
We will use $\Vol_{\mathcal H}$ and $\dist_{\mathcal H}$ to denote volume and distance in \(\mathcal H\) respectively.

Let \(\mathcal{B}(r) \subset \mathcal H\) be the set of points at distance at most \(r\) from the compact set \(\mathcal H_0 \). Observe that to prove Proposition \ref{approxgrowthlemma} it is enough to show that the volume of $\mathcal{B}(r)$ is comparable to $r^k$.

For each \(i = 1,\ldots, d\) and $t>0$,
we consider
\begin{equation}\label{approxball}
\mathcal H (t) = \mathcal H \cap \lbrace y \ge -t\rbrace\text{ and } \mathcal{H}_i^{\pm}(t) = \mathcal{H}_i^{\pm} \cap \mathcal{H}(t).
\end{equation}
The proof of Proposition \ref{approxgrowthlemma} will follow from volume estimates for the sets \(\mathcal{H}(t)\) and a comparison of these sets with \(\mathcal{B}(r)\) for suitable \(r\).


\begin{lemma}\label{volsublema}
 There exists \(C_1 > 1\) such that
\begin{equation}\label{mathcalH_i}
C_1^{-1}e^{\frac{\lambda_1+\cdots + \lambda_d}{2}t} \le \Vol_{\mathcal H} (\mathcal H_i^\pm(t))
  \le C_1e^{\frac{\lambda_1+\cdots + \lambda_d}{2}t},
\end{equation}
 for all \(t > 1\) and $i=1,\ldots,d$.
\end{lemma}

\begin{proof}
We prove (\ref{mathcalH_i}) for \(\mathcal H_1^+(t)\), the other cases can be obtained by analogous arguments.
For that purpose we parametrize \(\mathcal H_1^+ (t)\) via the map 
$\varphi : D(t) \rightarrow \mathcal H _1^\pm$,
defined by
\[\varphi(y,x_2,\ldots,x_d) = \bigl(y, e^{\frac{\lambda_1 y}{2}},x_2,\ldots,x_d\bigr),\]
where
\[D(t) = \left\lbrace y\in[-t,0]:-e^{\frac{\lambda_i y}{2}} \le x_i \le e^{\frac{\lambda_i y}{2}}\text{ for }2 \le i \le d\right\rbrace\]
is the domain of parametrization of $\mathcal H_1^+(t)$.
The induced metric on $\mathcal H_1^+(t)$ writes in these coordinates
\[\varphi^*g = \left(1+\frac{\lambda_1^2}{4}e^{-\lambda_1y} \right)dy^2 +e^{-2\lambda_2 y}dx_2^2 + \cdots + e^{-2\lambda_d y}dx_d^2 ,\]
and the associated volume form is
\[\omega = e^{-(\lambda_2+\cdots +\lambda_d)y}\left(1+\frac{\lambda_1^2}{4}e^{-\lambda_1y}\right)^{\frac{1}{2}} dy\wedge dx_2 \wedge \cdots \wedge dx_d.\]
Integrating we obtain
\begin{align*}\Vol_\mathcal H(\mathcal H_1^+(t)) &= \iint\limits_{D(t)}\omega = 2^{d-1}\int\limits_{-t}^0  e^{-\frac{\lambda_2 + \cdots +\lambda_d}{2}y}\left(1+\frac{\lambda_1^2}{4}e^{-\lambda_1y}\right)^{\frac{1}{2}}dy
\\ &= 2^{d-1}\int\limits_{-t}^0  e^{-\frac{\lambda_1 + \cdots +\lambda_d}{2}y}\left( e^{\lambda_1 y}+\frac{\lambda_1^2}{4}\right)^{\frac{1}{2}}  dy.
\end{align*}
The existence of a constant $C_1$ as in \eqref{mathcalH_i} is a consequence of the fact that \(e^{\lambda_1y} + \frac{\lambda_1^2}{4}\) is uniformly bounded away from \(0\) and \(+\infty\) when \(y \le 0\).
\end{proof}

The following two Lemmas aim to compare the sets $\mathcal H (t)$, 
$\mathcal H _1^+(t)$ and $\mathcal{B}(r)$.

\begin{lemma}\label{lowerlemma}
There exists \(C_2 > 0\) such that for all \(t > 0\) one has 
\begin{equation}\label{inc1}
\mathcal H_1^+(t) \subset \mathcal{B}\left(e^{\frac{\lambda_1(t+C_2)}{2}}\right).    
\end{equation}
\end{lemma}
\begin{proof}

Fix \((-u,x) \in \mathcal H_1^+(t)\), with
$u\leq t$,
and consider the curve $\alpha:[-u,0] \to \mathcal H_1^+(t)$ that joins this point to \(\mathcal H_1^+(0)\), defined by
 \[\alpha(s) = \bigl(s, e^{\frac{\lambda_1 s}{2}} ,x_2,\ldots,x_d\bigr).\] We have \[\|\alpha'(s)\|^2 = \frac{\lambda_1^2}{4}e^{-\lambda_1s} +1.\] 
It follows that there exists \(c > 1\) such that
 \[\dist\bigl((-u,x),\mathcal H_1^+(0)\bigr) \le \text{length}(\alpha) \le  u + \frac{\lambda_1}{2}\int_{-u}^0 e^{-\frac{\lambda_1 s}{2}}ds \le ce^{\frac{\lambda_1 t}{2}},\]
 for all \(1\leq u \leq t\).
 
 This establishes the inclusion \eqref{inc1}  for \(C_2 = 2\lambda_1^{-1}\log(c)\).
\end{proof}

\begin{lemma}\label{upperlemma}
There exists \(C_3 > 0\) such that
\begin{equation}\label{inc2}
\mathcal{B}\left(e^{\frac{\lambda_1(t-C_3)}{2}}\right) \subset \mathcal H(t)    
\end{equation}
for all $t > 1$.
\end{lemma}

\begin{proof}
We will show that there exists a constant $c$ such that, for every face $\mathcal H_i^\pm$, we have
$$
\mathcal{B}\left(e^{\frac{\lambda_1(t-c)}{2}}\right)\cap \mathcal{H}_i^\pm \subset \mathcal H_i^\pm(t).
$$
For that purpose we consider the function
$F :M \rightarrow \mathbb R$, defined by $F(y,x) = y$, and compute, for each $1\leq i \leq d$, an upper bound of the gradient of \(F\) restricted to \(\mathcal H_i^{\pm}\). This will give a lower bound for the distance between $\mathcal H_i^\pm (0)$ and $\mathcal H_i ^\pm \backslash \mathcal H_i^\pm (t)$ 
 in terms of $t$.

Without loss of generality we consider the case of $\mathcal H _i^+$.  We write $(y,\hat x):=(y,x_1,\ldots,\hat{x}_i,\ldots,x_d)$ to denote a point in \(\R^{d}\) obtained by removing the coordinate \(x_i\) from \((y,x)=(y,x_1,\ldots,x_d)\).
We parametrize the set \(\mathcal H_i^+\) via the map $\varphi : \mathbb R^d \rightarrow M$ defined by
\[\varphi(y,x_1,\ldots,\hat{x}_i,\ldots,x_d) = \left(y,x_1,\ldots,e^{\frac{\lambda_i y}{2}},\ldots,x_d\right).\]
The pull-back metric writes
\[\varphi^*g = \left(1+\frac{\lambda_i^2}{4}e^{-\lambda_iy} \right)dy^2+ \sum\limits_{j \neq i} e^{-2\lambda_j y}dx_j^2 .\]
Observe that
$$ \left\{\left( 1+\frac{\lambda _j^2}{4}e^{-\lambda _j y}\right)^{-1/2}\partial _y,\, \partial _{x_1}, ...,\, \partial _{x_{i-1}}, \partial _{x_{i+1}},..., \,\partial _{x_{d}} \right\}$$ is an orthonormal basis at $(y,\hat x)$ for the metric $\varphi ^* g$.

It follows that for every $(y,\hat x) \in \mathbb R^d$, we have $F\circ \varphi (y,\hat x)=y$ and
\begin{align*}
\bigl\|\nabla (F\circ \varphi)(y,\hat x)\bigr\|_{\varphi ^*g}^2 &= \left(1 + \frac{\lambda_i^2}{4}e^{-\lambda_i y}\right)^{-1}+ \sum_{j\neq i} 
\left(\partial _{x_j} (F\circ \varphi)(y, \hat x)\right)^2
\\ &= \left(e^{\lambda_iy} + \frac{\lambda_i^2}{4}\right)^{-1}e^{\lambda_i y}
\\ &\le C^2 e^{\lambda_1 y},
\end{align*}
for \(C = 2/\lambda_1\), all \(i\) and all \(y \le 0\).
We therefore have shown that for every face \(\mathcal H_i^\pm\), one has 
$$\bigl\|\nabla F|_{\mathcal H_{i}^{\pm}}\bigr\| \le C e^{\frac{\lambda_1 y}{2}}
$$
at all points of $\mathcal H_{i}^{\pm}$.

This implies that $F$ restricted to $\mathcal H (n+1)\setminus \mathcal H (n)$ is Lipschitz with constant \(Ce^{-\frac{\lambda_1 n}{2}}\) and thus,
\[\dist\mathcal{H}\bigl(\mathcal H(0), \mathcal H \setminus \mathcal H (n+1)\bigr) \ge \dist_{\mathcal H}\bigl(\mathcal H(n), \mathcal H \setminus \mathcal H(n+1)\bigr) \ge C^{-1}e^{\frac{\lambda_1 n}{2}}.\]
Therefore, we have \eqref{inc2} for all \(t > 1\) and a sufficiently large \(C_3 > 0\).
\end{proof}

We can now end the proof of Proposition 
\ref{approxgrowthlemma}.

\begin{proof}[Proof of Proposition \ref{approxgrowthlemma}]
On the one hand, observe that by Lemma \ref{lowerlemma} we have
\[\mathcal H_1^+(t) \subset  \mathcal{B}\left(e^\frac{\lambda_1(t+C_2)}{2}\right)\]
for all \(t > 1\), which by Lemma \ref{volsublema} implies
\[C_1^{-1}e^{\frac{\lambda_1+\cdots + \lambda_d}{2}t} \le \Vol_{\mathcal H}\left(\mathcal{B}\left(e^{\frac{\lambda_1(t+C_2)}{2}}\right)\right)\]
for all \(t > 1\).
Hence we obtain
\begin{equation}\label{lowervol}
\Vol_\mathcal H\bigl(\mathcal{B}(r)\bigr) \ge C_1^{-1}e^{-\frac{(\lambda_1+\cdots + \lambda_d) C_2}{2}}r^k
\end{equation}
for all \(r > e^{\frac{\lambda_1(1+C_2)}{2}}\).

On the other hand, we deduce from Lemma \ref{upperlemma} that
\[\mathcal{B}\left(e^{\frac{\lambda_1(t-C_3)}{2}}\right) \subset \mathcal H(t),\]
for all \(t > 1\), which by Lemma \ref{volsublema}, implies
\[\Vol_\mathcal H\left(\mathcal{B}\left(e^{\frac{\lambda_1(t-C_3)}{2}}\right)\right) \le C_1 e^{\frac{\lambda_1 + \cdots + \lambda_d}{2} t},\]
for all \(t > 1\). From this it follows that
\begin{equation}\label{uppervol}
Vol_\mathcal H\bigl(\mathcal{B}(r)\bigr) \le C_1 e^{\frac{(\lambda_1+ \cdots +\lambda_d)C_3}{2}} r^k
\end{equation}
for all \(r > e^{\frac{\lambda_1(1-C_3)}{2}}\). 

The inequalities (\ref{lowervol}) and (\ref{uppervol}) conclude the proof.
\end{proof}

\subsection{The approximate horosphere has controlled volume}
In order to compare the volume growth of a horosphere centered at $\xi_-$ with the one of $\mathcal H$, we will need the following homogeneity property: we say that a metric measure space $(X, d, \mu)$ has controlled volume if 
for each $r>0$ there exists $0<v(r)<V(r)$ such that 
\begin{equation}\label{controlledvol}
v(r)\leq \mu(B(p,r)) \leq V(r),
\end{equation}
for all $p\in X$.

\medskip
In what follows, we show that $\bigl(\mathcal H,dist_\mathcal{H},\Vol_\mathcal{H}\bigr)$ has controlled volume.
The idea is that, for every fixed $r>0$, the balls 
$B_\mathcal H (p,r)$ become arbitrarily close to the horizontal horosphere containing $p$
when $p$ tends to $\xi_-$ and hence, its volume is close to the volume of an Euclidean ball of the same radius.

\begin{proposition}\label{lemmaDoubling}
For all $\epsilon>0$ and $r>0$ there exists $y(\epsilon,r)<0$ such that for every $p=(y_p,x_p)\in \mathcal H$ with $y_p\leq y(\epsilon,r)$ we have
\begin{equation}\label{eqlemmaDoubling}
\left|\frac{\Vol_\mathcal{H} \bigl(B_\mathcal H(p,r)\bigr)}{\omega_d r^d}-1\right|\leq \epsilon,
\end{equation}
where $\omega_d$ is the volume of the unitary ball in the Euclidean space $\R^d$. 
\end{proposition}

As a consequence of Proposition \ref{lemmaDoubling} we have what we wanted:

\begin{corollary}\label{corollarycontvol}
The approximate horosphere $\mathcal{H}$ has controlled volume.
\end{corollary}

\begin{proof}
Let us fix $\epsilon = 1/2$ and $r>0$. By Proposition \ref{lemmaDoubling}, for every $p=(y_p,x_p)\in \mathcal H$ with $y_p\leq y(1/2,r)$ we have
\begin{equation}\label{ineq-1}
    \frac{\omega _d r^d}{2} \leq
    \Vol_\mathcal{H} \bigl(B_{\mathcal H} (p,r)\bigr) \leq
    \frac{3\omega _d r^d}{2}.
\end{equation}
The set $\mathcal H \cap \bigl\{p=(y,x) \in M : y(1/2,r) \leq y \leq 0 \bigr\}  $ is compact, hence the volume
of balls of radius $r$ centered on it satisfies
\begin{equation}\label{ineq-2}
    v_1 (R) \leq
    \Vol_\mathcal{H} B_{\mathcal H} (p,R) \leq
    v_2 (R)
\end{equation}
for some $v_1(r)$ and $v_2(r)$. 

We finish the proof by taking
$$v(r) := \min \left\{v_1 (r), \frac{\omega _d r^d}{2}\right\}\text{ and }V(r) := \max\left\{v_2 (r), \frac{3\omega _d r^d}{2}\right\}.$$
\end{proof}

To prove Proposition \ref{lemmaDoubling} we need the two following lemmas saying that a ball $B_{\mathcal H}(p,r) \subset \mathcal H$ becomes almost horizontal when $p$ tends to $\xi _{-}$. 


\begin{lemma}\label{lemmaDoub1}
For every $r>0$ there exists $C(r)>0$ such that if $p=(y_p, x_p)\in \mathcal H$ with $y_p\leq -R$ and $q=(y,x)\in B_{\mathcal H}(p,r)$, then
$$|y-y_p|\leq C(r)e^{\frac{\lambda_1 y_p}{2}}.$$
In particular, 
$\bigl|b_{\xi^+} (q) - b_{\xi^+} (p)\bigr| \leq C(r)e^{\frac{\lambda_1 y_p}{2}}.$
\end{lemma}

\begin{proof} Let $\gamma:[0,1]\to \mathcal{H}$ be a geodesic from $p$ to $q$ with $\length(\gamma)=\dist_{\mathcal H}(p,q)\leq r$. We write $\gamma(t)=\bigl(y(t),x_1(t),\ldots,x_d(t)\bigr)$.

Decompose $\gamma$ into $\gamma_j:[t_j,t_{j+1}]\to \mathcal H_{i_j}^\pm$, where 
$$\gamma_j(t)=\left(y(t), x_1(t),\ldots,x_{i_j-1}(t),\pm e^{\frac{\lambda_{i_j}y(t)}{2}},x_{i_j+1}(t),\ldots,x_d(t) \right).$$
The length of $\gamma_j$ satisfies
$$\length(\gamma_j)\geq \int_{t_j}^{t_{j+1}}\frac{\lambda_{i_j}|\dot{y}(t)|}{2}e^{-\frac{\lambda_{i_j}y(t)}{2}}\,dt.$$
Approximating $y(t)$ by a Morse function, we can assume without loss of generality that each interval $[t_j,t_{j+1}]$ subdivides into sub-intervals $I_{jk}=[s_k^j,s_{k+1}^j]$ where $\dot{y}(t)$ has constant sign. We then have 
$$\length\left(\gamma_j|_{I_{jk}}\right)\geq \pm\int_{s^j_k}^{s^j_{k+1}}\frac{\lambda_{i_j}\dot{y}(t)}{2}e^{-\frac{\lambda_{i_j}y(t)}{2}}\,dt,$$
and setting $u=y(t)$ we get
$$\length\left(\gamma_j|_{I_{jk}}\right)\geq \left|e^{-\frac{\lambda_{i_j}y\left(s^j_{k+1}\right)}{2}}-e^{-\frac{\lambda_{i_j}y\left(s^j_{k}\right)}{2}}\right|.$$

Assume for example $y\bigl(s_k^j\bigr)\leq y\bigl(s_{k+1}^j\bigr)\leq 0$, then
\begin{align*}
\length(\gamma_j|_{I_{jk}}) &\geq e^{-\frac{\lambda_{i_j}y\left(s^j_{k}\right)}{2}}-e^{-\frac{\lambda_{i_j}y\left(s^j_{k+1}\right)}{2}}\\
&\geq e^{-\frac{\lambda_{i_j}y\left(s^j_{k+1}\right)}{2}}\left( e^{\frac{\lambda_{i_j}y\left(s^j_{k+1}\right)}{2}-\frac{\lambda_{i_j}y\left(s^j_{k}\right)}{2}} -1 \right)\\
&\geq e^{-\frac{\lambda_{i_j}y\left(s^j_{k+1}\right)}{2}}\frac{\lambda_{i_j}}{2}\left(y\bigl(s^j_{k+1}\bigr)-y\bigl(s^j_{k}\bigr)\right)\\
&\geq e^{-\frac{\lambda_1(y_p +r)}{2}}\frac{\lambda_1}{2}\left(y\bigl(s^j_{k+1}\bigr)-y\bigl(s^j_{k}\bigr)\right)
\end{align*}
The other case is analogous, then we get
$$\left|y\bigl(s^j_{k+1}\bigr)-y\bigl(s^j_{k}\bigr)\right|\leq 2\lambda_1^{-1}e^{\frac{\lambda_1 (y_p +r)}{2}}\length(\gamma_j|_{I_{jk}});$$
hence,
\begin{align*}
|y_p-y|=\left|\sum_{j,k} y\bigl(s^j_{k+1}\bigr)-y\bigl(s^j_{k}\bigr)\right| &\leq \sum_{j,k}\left|y\bigl(s^j_{k+1}\bigr)-y\bigl(s^j_{k}\bigr)\right| \\
&\leq 2\lambda_1^{-1}e^{\frac{\lambda _1 r}{2}}
e^{\frac{\lambda_1y_p}{2}}\length(\gamma).
\end{align*}
That is,
\begin{equation}\label{lD12}
|y_p-y|\leq 2\lambda_1^{-1}e^{\frac{\lambda _1 r}{2}}
e^{\frac{\lambda_1y_p}{2}}r,
\end{equation}
which concludes the proof by taking $C(r) = 2\lambda_1^{-1} re^{\frac{\lambda _1 r}{2}} $.
\end{proof}

Given $t\in\R$ we define
the projection
$\Pi_{t}: M\to H_{\xi_+}(t)$ by
\begin{equation*}\label{lD14}
\Pi_{t}(y,x)=(t,x).
\end{equation*}

\begin{lemma}\label{lD2}
For every $\epsilon>0$ and $r>0$ there exists $y(\epsilon,r)<0$ such that if $p=(y_p,x_p)$ with $y_p<y(\epsilon,r)$ and $q\in B_\mathcal H(p,r)$, then
$$(1-\epsilon)\|v\|_q \leq \| d\Pi_{y_p}(q)v\|_{\Pi_{y_p}(q)} \le (1+\epsilon)\|v\|_{q},$$
for all tangent vectors \(v \in T_p \mathcal H_i^{\pm}\), $i=1,\ldots,d$.  So in particular one also has
$$(1-\epsilon)^d \leq \Jac \Pi_{y_p}(q)\leq (1+\epsilon)^d.$$

\end{lemma}

\begin{proof}
Let $q=\Bigl(y,x_1,\ldots,x_{j-1},\pm e^{\frac{\lambda_jy}{2}},x_{j+1},\ldots,x_d\Bigr)\in \mathcal H^\pm _j$. The tangent spaces to $\mathcal H^\pm _j$ at $q$ is generated by $\partial_{x_1},\ldots,\widehat{\partial}_{x_j},\ldots,\partial_{x_d}$ and $\partial_y \pm e^{\frac{\lambda_jy}{2}}\partial_{x_j}$. Denote $X_i=\partial_{x_i}$ and $Z_j=\partial_y \pm\frac{\lambda_j}{2} e^{\frac{\lambda_jy}{2}}X_j$.

We have 
$$d\Pi_{y_p}(q)X_i=X_i\text{ for every }i\neq j\text{ and } d\Pi_{y_p}(q)Z_j= \pm \frac{ \lambda_j}{2} e^{\frac{\lambda_jy}{2}}X_j.$$
Notice that $X_1,\ldots,\widehat{X}_j,\ldots,X_d,Z_j$ are orthogonal and so are $X_1,\ldots,X_j,\ldots,X_d$. Moreover, 
$$\|X_i\|_q=e^{-\lambda_iy},\ \|X_i\|_{\Pi_{y_p}(q)}=e^{-\lambda_iy_p},\text{ and }\|Z_j\|_q=\left(1+\frac{\lambda_j^2}{4}e^{-\lambda_jy} \right)^{\frac{1}{2}};$$
thus,
\begin{equation}\label{isom1}\frac{\|d\Pi_{y_p}(q)X_i\|_{\Pi_{y_p}(q)}}{\|X_i\|_q} = e^{\lambda_i(y-y_p)},\end{equation}
for all \(i \neq j\), and
\begin{equation}\label{isom2}\frac{\|d\Pi_{y_p}(q)Z_j\|_{\Pi_{y_p}(q)}}{\|Z_j\|_q} = \frac{\frac{\lambda_j}{2} e^{\frac{\lambda_jy}{2}}\|X_j\|_{\Pi_{y_p}(q)}}{\|Z_j\|_q} = \frac{\frac{\lambda_j}{2} e^{\frac{\lambda_jy}{2}}e^{- \lambda_j y_p}}{(1+ \frac{\lambda_j^2}{4}e^{-\lambda_j y})^{\frac{1}{2}}}.\end{equation}

By Lemma \ref{lemmaDoub1}, the right-hand sides of inequalities \ref{isom1} and \ref{isom2} go to \(1\) as \(y_p \to -\infty\).  Hence, the singular values of \(d\Pi_{y_p}(q)\) can be made arbitrarily close to \(1\) by taking \(y_p\) large, which concludes the proof.
\end{proof}

\begin{proof}[Proof of Proposition \ref{lemmaDoubling}]
Given $\epsilon,r>0$, let \(p = (y_p,x_p)\) be as in Lemma \ref{lD2} and obtain
\begin{equation}\label{lD16}
B_{H_{\xi_+}(y_p)}\bigl(\Pi_{y_p}(p),(1-\epsilon)r\bigr)\subset \Pi_{y_p}\bigl(B_\mathcal H(p,r)\bigr)\subset B_{H_{\xi_+}(y_p)}\bigl(\Pi_{y_p}(p),(1+\epsilon)r\bigr),
\end{equation}
and furthermore
\[(1-\epsilon)^d \le \frac{\Vol_\mathcal{H}\bigl(B_\mathcal H(p,r)\bigr)}{\Vol_{H_{\xi_+}(y_p)}\bigl(\Pi_{y_p}\bigl(B_\mathcal H(p,r)\bigr)\bigr)} \le (1+\epsilon)^d.\]

Hence,
\[(1-\epsilon)^d \omega_d \bigl((1-\epsilon)r\bigr)^d \le \Vol_\mathcal{H}\bigl(B_\mathcal H(p,r))\bigr) \le (1+\epsilon)^d\omega_d \bigl((1+\epsilon)r\bigr)^d,\]
which ends the proof.
\end{proof}

\section{Quasi-isometry}
So far, we have constructed an approximate horosphere $\mathcal H$ and computed its volume growth. Our goal now is to prove that $\mathcal{H}$ have the same volume growth as the actual horosphere  
$H_t:=H_{\xi_-}(t)$ for any $t\in\R$. In order to do this, we first show that $H_t$ and  $\mathcal H$ are quasi-isometric. Recall that a map between metric spaces $\varphi : (X, d_1) \rightarrow (Y, d_2)
$
is a quasi-isometry if there exist \(a\geq 1\) and $b\geq 0$ such that 
\begin{equation}
\label{quasiisom}
a^{-1}d_1(x,x') - b \le d_2(\varphi(x),\varphi(x')) \le a d_1(x,x') + b
\end{equation}
for all $x,x' \in X,$
and for every $y\in Y$ there exists $x\in X$ such that $d_2(y,\varphi(x))\leq b$.

\begin{proposition}\label{quasi-isometry-prop}
For every $t\in \mathbb R$, the horosphere $H_t$ is quasi-isometric to the approximate horosphere $\mathcal H$.
\end{proposition}
The proof of Proposition \ref{quasi-isometry-prop} relies on projection arguments on convex subset of $M$ that we describe in what follows.


We consider the sets
$$\hat{V}=V\cap HB_{\xi_-}(C_V)\ \text{ and }\ \hat{\mathcal{H}}=\partial \hat{V},$$
where $C_V$ is the constant of Lemma \ref{approxhorolemma}. Since $\hat{\mathcal{H}}$ and $\mathcal{H}$ differ on a compact set, they are clearly quasi-isometric. Notice that $\hat{\mathcal{H}}$ lies outside $HB_{\xi_-}(-C_V)$ and that $H_{\xi_{-}} (C_V)$ lies outside the interior of $\hat{V}$. 

By Lemma \ref{vconvexitylemma} and the convexity of the horoballs, the sets $HB_{\xi_-}(-C_V)$ and $\hat{V}$ are convex. Therefore, by
\cite[Proposition 2.4, page 176]{bridson-haefliger}, the orthogonal projections 
$\varphi_1:H_{C_V} \to \hat{\mathcal{H}}$ and 
$\varphi_2: \hat{\mathcal{H}} \to H_{-C_V}$ are well defined and $1$-Lipschitz since $M$ is negatively curved.

\begin{lemma}\label{lipschitzlemma}
  The orthogonal projection \(\varphi_1:H_{C_V} \to \hat{\mathcal{H}}\) and the orthogonal projection \(\varphi_2: \hat{\mathcal{H}} \to H_{-C_V}\) are surjective and \(1\)-Lipschitz.
\end{lemma}
\begin{proof}
It remains to show the surjectivity of \(\varphi_1\). We take \(p \in \hat{\mathcal{H}}\). By \cite[Proposition 2.4, (4) page 176]{bridson-haefliger}, there exists a unit speed ray \(\alpha\) starting at \(p\) such that 
$$\dist \bigl(\alpha (t), \hat{V}\bigr)=\dist\bigl(\alpha(t),\hat{\mathcal{H}}\bigr) = \dist(\alpha(t),p) = t$$
and the projection of $\alpha (t)$ on $\hat{\mathcal{H}}$ equals $p$ for every $t\geq 0$.
In particular, when $t\to +\infty$, the distance between $\alpha (t)$ and every geodesic ray asymptotic to $\xi_-$ tends to infinity, thus $\alpha (t)$ converge to $\xi' \neq \xi _{-} \in \partial M$ when $t\to +\infty$. It follows that \( \lim _{t\to \infty} b_{\xi_{-}}(\alpha(t)) = +\infty\). Since 
$b_{\xi_{-}}(\alpha(0)) = b_{\xi_{-}}(p) \leq C_V$, there exists $t\geq 0$ such that $b_{\xi_{-}}(\alpha(t)) = C_V$ and therefore $\alpha (t) \in H_{C_V}$. This proves that \(\varphi_1\) is surjective. Moreover, since the Hausdorff distance between $H_{C_V}$ and $\hat{\mathcal{H}}$ is bounded by $2C_V$, then $t\leq 2C_V$.

For \(\varphi_2\) we argue similarly.  Given \(p \in H_{-C_V}\) we let \(\alpha(t)\) be the (now unique) unit speed geodesic ray with \(\alpha(0) = p\) and \(b_{\xi_-}(\alpha(t)) = t-C_V\). We observe that \(\alpha(2C_V) \in H_{C_V}\) and therefore by Lemma \ref{approxhorolemma} we must have \(\alpha(t) \in \hat{\mathcal{H}}\) for some \(t \in [0,2C_V]\).   Hence \(\varphi _2\) is surjective as claimed.
\end{proof}

\begin{lemma}\label{qitricklemma}
 Let \(\varphi_3:H_{C_V} \to H_{-C_V}\) be the orthogonal projection.  Then \(\varphi_3\) is bi-Lipschitz and there exists \(C > 0\) such that
 \[\dist_{H_{C_V}}\bigl((\varphi_3^{-1}\circ \varphi_2 \circ \varphi_1)(p),p\bigr) \le C,\]
for every $p\in H_{C_V}$.
\end{lemma}

\begin{proof}
We have that \(\varphi_3\) is \(1\)-Lipschitz since its an orthogonal projection onto a convex domain in a space of non-positive sectional curvature.

To show that \(\varphi _3^{-1}\) is Lipschitz, let \(p(t)\) be unit speed curve in \(H_{C_V}\) and \(\alpha_t(s)\) be the geodesic perpendicular to \(H_{C_V}\) such that \(\alpha_t(0) = p(t)\) and \(\alpha_t(2C_V) = \varphi_3(p(t))\).

Using the comparison theorem for stable Jacobi fields \cite[Theorem 2.4]{1977heintze}, we obtain that 
\[\left\|\frac{d}{dt}(\varphi_3\circ p)(0)\right\|= \left\|\frac{d}{dt} \alpha_0(2C_V)\right\| \ge e^{-2\lambda_d C_V} \left\|\frac{d}{dt} \alpha_0(0)\right\| = e^{-2\lambda_d C_V}\|\dot{p}(0)\|,\]
so that \(\varphi_3^{-1}\) is Lipschitz with constant $e^{2\lambda_d C_V}$.

We now prove the existence of \(C > 0\) with the desired properties.

Set $p_1 = p$,  $p_2 = \varphi_1(p_1)$,  $p_3= \varphi_2(p_2)$, and $p_4 = \varphi_3^{-1}(p_3)$. As in the previous lemma $p_i$ and $p_{i+1}$ are endpoints of a geodesic segment of length bounded by \(2C_V\) for \(i = 1,2\). Moreover, $p_2$ and $p_3$ are endpoints of a geodesic segment of length $2C_V$; thus, by triangular inequality, we have
 \[\dist\bigl((\varphi_3^{-1}\circ \varphi_2 \circ \varphi_1)(p),p\bigr) \le 6C_V.\]

From \cite[Theorem 4.6]{1977heintze} this implies
\[\dist_{H_{C_V}}\bigl((\varphi_3^{-1}\circ \varphi_2 \circ \varphi_1)(p),p\bigr) \le \frac{2}{\lambda_d}\sinh\left(\frac{\lambda_d}{2}6C_V\right),\]
where \(-\lambda_d^2\) is a lower bound for sectional curvature on \(M\).
\end{proof}

\begin{proof}[Proof of Proposition \ref{quasi-isometry-prop}]
Notice that for every $t>t' \in \mathbb R$, the projection of $H_{t}$ onto $H_{t'}$ is bi-Lipschitz. The proof of this fact is the same as the one saying that $\varphi _3$ is bi-Lipschitz in Lemma \ref{qitricklemma}. Therefore, all horospheres centered at $\xi_-$ are quasi-isometric, so it suffices to prove that 
$H_{C_V}$ is quasi-isometric to $\mathcal H$.

Let us prove that the map $\varphi _1 : H_{C_V} \rightarrow \hat{\mathcal{H}}$
is a quasi-isometry. By Lemma \ref{lipschitzlemma}, $\varphi _1$ is surjective, hence it is enough to prove that there exist \(a,b > 0\) such that
 \[a^{-1}\dist_{H_{C_V}}(p,q) - b \le \dist_{\hat{\mathcal{H}}}(\varphi_1(p),\varphi_1(q)) \le \dist_{H_{C_V}}(p,q),\]
 for all \(p,q \in H_{C_V}\).

The upper bound is the \(1\)-Lipschitz property established in Lemma \ref{lipschitzlemma}. For the lower bound, let \(\lambda\) be the Lipschitz constant of \(\varphi_3^{-1}\) and \(C\) be given by Lemma \ref{qitricklemma}.
 
 Given \(p,q \in H_{C_V}\) let \(p' = (\varphi_3^{-1}\circ \varphi_2\circ \varphi_1)(p)\) and \(q' = (\varphi_3^{-1}\circ \varphi_2\circ \varphi_1)(q)\). We have from Lemma \ref{lipschitzlemma} and Lemma \ref{qitricklemma} that
 \[\dist_{H_{C_V}}(p,q) \le 2C + \dist_{H_{C_V}}(p',q') \le 2C + \lambda \dist_{\hat{\mathcal{H}}}(\varphi_1(p),\varphi_1(q)),\]
 which establishes the required lower bound.
\end{proof}

\section{Isometry at infinity}

The goal of this section is to prove Theorem \ref{mainthm}, i.e., that the growth of the non Euclidean horospheres is of order 
$k=(\lambda _1 +...+ \lambda _d)/\lambda _1$. By Proposition \ref{two-isom-class}, all horospheres centered at $\xi' \neq \xi^+$ are isometric to each other, hence it suffices to show that $\mathcal H$ and $H_{C_V}$ have growth of same order and use Proposition \ref{approxgrowthlemma}. 

The proof that $\mathcal H$ and $H_{C_V}$ have volume growth of same order rely on a result of Coulhon-Saloff-Coste that we describe now. In their setting, a metric measure space \((X,d,\mu)\) is said to be \textit{locally doubling} if for each \(r > 0\) there exists \(C_r > 0\) such that
\begin{equation}\label{doubling}
\mu\bigl(B(x,2r)\bigr) \le C_r\mu\bigl(B(x,r)\bigr),
\end{equation}
for all \(x \in M\). 

The following notion of isometry at infinity between measured metric spaces was introduced by Kanai, 
\cite{Kanai,1995coulhon}.

\medskip
\begin{definition}\label{def-isom-inf}
A mapping $\varphi : (X,d_1,\mu)\rightarrow
(Y,d_2,\nu)$
between two metric measured spaces 
is said to be an {\it isometry at infinity} if it is a quasi-isometry and there exists \(C > 0\) such that 
$$C^{-1}\mu\bigl(B(x,1)\bigr) \le \nu\bigl(B(\varphi(x),1)\bigr) \le C \mu\bigl(B(x,1)\bigr)$$ 
for every \(x \in X\).

\end{definition}

The following result is the Proposition 2.2 of \cite{1995coulhon}.

\begin{proposition}\label{coulhonlemma}
Let \((X,d_1,\mu)\) and \((Y,d_2,\nu)\) be two locally doubling metric measure spacesand \(\varphi:X \to Y\) be an isometry at infinity.  Then there exists \(C \geq 1\) such that
\[C^{-1} \mu\bigl(B(x,C^{-1}r)\bigr) \le \nu\bigl(B(\varphi(x),r)\bigr) \le C \mu\bigl(B(x,Cr)\bigr),\]
for all \(x \in X\) and \(r \ge 1\).
\end{proposition}

Notice that if a metric measure space has controlled volume, then it is locally doubling. Moreover, if \((X,d_1,\mu)\) and \((Y,d_2,\nu)\) have controlled volume for functions $v_X,V_x$ and $v_Y,V_Y$ respectively, then for every $x\in X$, $y\in Y$ and $r\geq 0$ we have
$$\frac{v_Y(1)}{V_X(1)}\mu\bigl(B(x,1)\bigr) \le \nu\bigl(B(y,1)\bigr) \le \frac{V_Y(1)}{v_X(1)} \mu\bigl(B(x,1)\bigr).$$

This implies the following result:

\begin{corollary}\label{Cor3}
Volume growth is a quasi-isometry invariant between metric measure spaces with controlled volume.    
\end{corollary}

Since we want to apply the preceding to $H_{C_V}$ and $\mathcal{H}$, we need the following lemma:

\begin{lemma}\label{cont-vol-H}
The horosphere $H_{C_V}$ has  controlled volume.
\end{lemma}

\begin{proof}
Since all horospheres centered at $\xi_-$ are in the same isometry class it suffices to prove that for each 
\(\epsilon>0\) and \(r>0\) there exist \(T = T(\epsilon,r) > 0\) and \(y_0 = y_0(\epsilon,r) < 0\) such that for every \(p=(y_p,x_p)\in H_{T}\) with \(y_p\leq y_0\) we have
\begin{equation}
\left|\frac{\Vol_{H_T} \bigl(B_{H_{T}}(p,r)\bigr)}{\omega_d r^d}-1\right|\leq \epsilon,
\end{equation}
where \(\omega_d\) is the volume of the unitary ball in the Euclidean space \(\R^d\). Then the proof ends by repeating the proof of Corollary \ref{corollarycontvol}.
\medskip

We denote by \(\alpha:\R \to \R\) the geodesic defined by \(\alpha(t) = (0,t)\) in \(M\). Given \(p=(x_p,y_p) \in M\) let \(\beta_{p}^+,\beta_{p}^-:[0,+\infty) \to M\) be the unit speed geodesic rays that are asymptotic to $\alpha$ to the future and the past respectively. By direct calculation with the metric \(g\) one observes that \(\beta_{p}^+(t) = (x_p,y_p+t)\).  

We first observe that, for each \(\epsilon' > 0\), there exists \(D > 0\) such that if \(p \in M\) is such that \(\dist(p, \alpha(\R)) \ge D\) then 
 \begin{equation}\label{geodesicview}\Bigl|g\bigl(\dot{\beta}_{p}^+(0),\dot{\beta}_{p}^-(0)\bigr)\Bigr| = \Bigl|g\bigl(\partial_y,\dot{\beta}_{p}^-(0)\bigr)\Bigr| \ge 1-\epsilon'.
 \end{equation}

To see this, let $D$ and $t_0$ be such that \(D = \dist(p,\alpha(\R)) = \dist(p,\alpha(t_0))\).   We consider a comparison triangle \(\overline{p},\overline{q},\overline{r}\) in the model space with constant curvature \(-\lambda_1^2\) such that \(\dist(\overline{p},\overline{q}) = D,  \dist(\overline{q},\overline{r}) = t, \dist(\overline{p},\overline{r}) = \dist(p,\alpha(t_0+t))\) (remember that $-\lambda_1^2$ is the upper bound on the sectional curvature of $M$).   By \cite[Proposition 1.7 part 1, page 169]{bridson-haefliger} every point in  the geodesic segment \([\overline{q},\overline{r}]\) is at distance at least \(D\) from \(\overline{p}\).   It follows that the angle at \(\overline{p}\) in the comparison triangle goes to \(0\) when \(D \to +\infty\) independently of \(t\).   By \cite[Proposition 1.7 part 4]{bridson-haefliger} this implies that the angle at \(p\) of the ideal triangle with vertices $p,\alpha(t_0)$ and $\xi_+$ is bounded from above by a function that goes to \(0\) when \(D \to +\infty\). The same argument applied to the vertices $p,\alpha(t_0)$ and $\alpha(t_0-t)$ concludes the proof of 
(\ref{geodesicview}).

Now let \(T = T(\epsilon,r) = D + r = -y_0(\epsilon,r) = -y_0\), and \(p = (y_p,x_p) \in H_T\) be such that \(y_p < y_0\).  Since 
$$\dist\bigl(p,\lbrace y \ge 0\rbrace\bigr) \ge |y_p| > D + r\text{, and }\dist\bigl(p,HB_{\xi_-}(0)\rbrace\bigr) \ge T \ge D+r,$$
we obtain \(\dist\bigl(p,\alpha(\R)\bigr) \ge D + r\),  and therefore 
 \begin{equation}\label{gooddist}\dist\bigl(q,\alpha(\R)\bigr) \ge D,\end{equation}
for all \(q \in B_{H_T}(p,r)\).
 
Let \(q = (y,x) \in B_{H_T}(p,r)\) and set \(\Pi_y:M \to H_{\xi_+}(y)\)  the orthogonal projection. From \ref{gooddist} and \ref{geodesicview} we obtain
 \[(1-\epsilon')\|v\|_q \le \|d\Pi_y(q)v\|_{q} \le (1+\epsilon')\|v\|_q,\]
 for all \(q \in B_{H_T}(p,R)\) and \(v \in T_qH_T\).
 
 Composing with the projection \(\Pi_{y_p}\) we obtain
 \[(1-\epsilon')e^{-2\lambda_d|y-y_p|}\|v\|_q \le \|d\Pi_{y_p}(q)v\|_{\Pi_{y_p}(q)} \le (1+\epsilon')e^{2\lambda_d|y-y_p|}\|v\|_q,\]
 for all \(q \in B_{H_T}(p,R)\) and \(v \in T_qH_T\).
 
 To conclude we consider a curve \(\gamma:[0,s] \to B_{H_T}(p,r)\) parametrized by arc-length.  Observe that \(\dot{\beta}_{\gamma(t)}^-(0)\) is a unit normal vector to \(B_{H_T}\) at \(\gamma(t)\) for all \(t\);  hence, by equation \eqref{geodesicview} 
 we obtain
 \[\bigl|y(s)-y(0)\bigr| \le \left|\int\limits_0^{s} g\bigl(\partial_y, \gamma'(t)\bigr) dt\right| \le \epsilon' s.\]
 
 We thus deduce for all \(q \in B_{H_T}(p,r)\) and \(v \in T_pH_T\) that
\[(1-\epsilon')e^{-2\lambda_d r \epsilon'}\|v\|_q \le \|d\Pi_{y_p}(q)v\|_{\Pi_{y_p}(q)} \le (1+\epsilon')e^{2\lambda_d r\epsilon'}\|v\|_q.\]

Therefore, taking \(\epsilon' > 0\) small enough we have
\[(1-\epsilon)^{\frac{1}{2d}}\|v\|_q \le \|d\Pi_{y_p}(q)v\|_{\Pi_{y_p}q} \le (1+\epsilon)^{\frac{1}{2d}}\|v\|_q.\]
  
The proof then finishes by repeating the argument of Proposition \ref{lemmaDoubling}.
\end{proof}

\begin{proof}[Proof of Theorem \ref{mainthm}]
As we say at the beginning of the section, it is enough to give the volume growth for the horosphere $H_{C_V}$.

Since $H_{C_V}$ and $\mathcal{H}$ have both controlled volume by Corollary \ref{corollarycontvol} and Lemma \ref{cont-vol-H}, and they are quasi-isometric by Proposition \ref{quasi-isometry-prop}, we can apply Corollary \ref{Cor3} to ensure that they have volume growth of same order. Finally, by Proposition \ref{approxgrowthlemma} we obtain that $H_{C_V}$ has volume growth of order $r^k$.
\end{proof}

Since an Euclidean horosphere in $M$ has controlled volume and volume growth of order $r^d$, it cannot be quasi-isometric to any non-Euclidean horosphere. This concludes Corollary \ref{quasiisom-cor}.


\bibliographystyle{alpha}
\bibliography{biblio}
\end{document}